\documentclass[12pt,a4paper,psamsfonts]{amsart}
\usepackage{amssymb,amscd,amsxtra,calc}
\usepackage{cmmib57}

\theoremstyle{plain}
    \newtheorem{thm}{Theorem}[section]
    
    \newtheorem{claim}[thm]{Claim}
    \newtheorem{corollary}[thm]{Corollary}
    \newtheorem{example}[thm]{Example}
    \newtheorem{lemma}[thm]{Lemma}
    \newtheorem{proposition}[thm]{Proposition}
    \newtheorem{question}[thm]{Question}
    \newtheorem{theorem}[thm]{Theorem}

\theoremstyle{definition}

    \newtheorem{remark}[thm]{Remark}
\theoremstyle{remark}

    \newtheorem{setup}[thm]{}

\newcommand{\BCC}{\mathbb{C}}

\newcommand{\PP}{\mathbb{P}}
\newcommand{\BPP}{\mathbb{P}}
\newcommand{\BQQ}{\mathbb{Q}}
\newcommand{\Q}{\mathbb{Q}}
\newcommand{\R}{\mathbb{R}}
\newcommand{\BRR}{\mathbb{R}}
\newcommand{\BZZ}{\mathbb{Z}}

\newcommand{\OO}{\mathcal{O}}
\newcommand{\SO}{\mathcal{O}}

\newcommand{\diag}{\operatorname{diag}}

\newcommand{\id}{\operatorname{id}}

\newcommand{\NE}{\operatorname{NE}}

\newcommand{\NS}{\operatorname{NS}}

\newcommand{\Pic}{\operatorname{Pic}}
\newcommand{\rank}{\operatorname{rank}}
\newcommand{\Sing}{\operatorname{Sing}}
\newcommand{\Supp}{\operatorname{Supp}}

\newcommand{\isom}{\simeq}

\begin{document}

\title[Invariant hypersurfaces]{
Invariant hypersurfaces of endomorphisms of projective varieties}

\author{De-Qi Zhang}
\address
{
\textsc{Department of Mathematics} \endgraf
\textsc{National University of Singapore,
10 Lower Kent Ridge Road,
Singapore 119076
}}
\email{matzdq@nus.edu.sg}

\begin{abstract}
We consider surjective endomorphisms $f$ of degree $> 1$ on projective
manifolds $X$ of Picard number one
and their $f^{-1}$-stable hypersurfaces $V$, and
show that $V$ is rationally chain connected.
Also given is an optimal upper bound for the number of $f^{-1}$-stable prime divisors
on (not necessarily smooth) projective varieties.
\end{abstract}

\subjclass[2000]{14E20, 14J45, 37F10, 32H50}
\keywords{endomorphism, projective spaces, iteration, complex dynamics}

\thanks{The author is partially supported by an ARF of NUS}

\maketitle

\section{Introduction}

We work over the field $\BCC$ of complex numbers.
Theorems \ref{ThE'} $\sim$ \ref{ThG} below are our main results.
We refer to
\cite[Definition 2.34]{KM} for the definitions of {\it Kawamata log terminal} ($klt$) and
{\it log canonical singularities}.
See S.~-W. Zhang \cite[\S1.2, \S4.1]{Zs} for the
Dynamic Manin-Mumford conjecture solved
for the pair $(X, f)$ as in the conclusion part of Theorem \ref{ThE'} below,
and \cite{Ca} for a related result on endomorphisms of
(not necessarily projective) compact complex manifolds.

\begin{theorem}\label{ThE'}
Let $X$ be a normal projective variety of dimension $n \ge 2$,
$V_i$ $(1 \le i \le s)$ prime divisors, $H$ an ample Cartier divisor,
and $f : X \to X$ an endomorphism with $\deg(f) = q^n > 1$ such that
$($for all $i)$:
\begin{itemize}
\item[(1)]
$X$ has only log canonical singularities around $\, \cup \, V_i \,$;
\item[(2)]
$V_i$ is Cartier and $V_i \equiv d_iH$ $($numerically$)$ for some $d_i > 0$; and
\item[(3)]
$f^{-1}(V_i) = V_i$.
\end{itemize}
Then $s \le n+1$.
Further, the equality $s = n+1$ holds if and only if:
$$X = \BPP^n, \,\,\,\, \, V_i = \{X_i = 0\} \,\,\,\, (1 \le i \le n+1)$$
$($in suitable projective coordinates$)$,
and $f$ is given by
$$f: \,\, [X_1, \dots, X_{n+1}] \to [X_1^q, \dots, X_{n+1}^q].$$
\end{theorem}

The conditions (1) and (2) in Theorem \ref{ThE'} are satisfied
if $X$ is smooth with Picard number $\rho(X) = 1$.
The ampleness of $V_i$ in Theorem \ref{ThE'} above and
the related result Proposition \ref{ThF} below
(with the Cartier-ness of $V_i$ replaced by the weaker $\Q$-Cartier-ness)
is quite necessary because there are endomorphisms $f$ of degree $> 1$
on toric surfaces whose boundary divisors have as many irreducible components
as you like and are all stabilized by $f^{-1}$.
The condition (1) is used to guarantee the inversion of adjunction (cf. \cite{Kk})
and can be removed in dimension two (cf. \cite[Theorem B]{Fa}, \cite[Theorem 4.3.1]{ENS}).

A projective variety $X$ is {\it rationally chain connected} if every two points $x_i \in X$
are contained in a connected chain of rational curves on $X$.
When $X$ is smooth, $X$ is rationally {\it chain} connected if and only if $X$
is {\it rationally connected}, in the sense of Campana, and Koll\'ar-Miyaoka-Mori
(\cite{Cp}, \cite{KoMM}).

The condition (1) below is satisfied if $X$ is $\Q$-factorial with
Picard number $\rho(X) = 1$,
while the smoothness (or at least the mildness of singularities) of $X$
in (3) is necessary (cf.~Remark \ref{rThB}).

\begin{theorem}\label{ThB'}
\footnote{By the recent paper of
A.~Broustet and A.~Hoering ''Singularities of varieties admitting an endomorphism,''
arXiv:1304.4013, the condition (1) in Theorem \ref{ThE'}, condition (2) in Theorem \ref{ThB'}
and similar conditions in Propositions \ref{ThC} and \ref{ThF} are automatically satisfied if
$X$ is $\Q$-Gorenstein and has a polarized endomorphism of degree $> 1$.}
Let $X$ be a normal projective variety of dimension $n \ge 2$,
$f: X \to X$ an endomorphism of degree $> 1$, $(0 \ne)$ $V = \sum_i V_i \subset X$
a reduced divisor with $f^{-1}(V) = V$, and $H \subset X$ an ample Cartier divisor.
Assume the three conditions below $($for all $i)$:
\begin{itemize}
\item[(1)]
$-K_X \sim_{\BQQ} rH$ $(\Q$-linear equivalence$)$ and $V_i \sim_{\BQQ} d_iH$ for some
$r, d_i \in \Q$;
\item[(2)]
$X$ has only log canonical singularities around $V$; and
\item[(3)]
$X$ is further assumed to be smooth
if: $V = V_1$ $($i.e., $V$ is irreducible$)$, $K_X + V \sim_{\BQQ} 0$ and
$f$ is \'etale outside $V \cup f^{-1}(\Sing X)$.
\end{itemize}
Then $X$, each irreducible component $V_i$ and the normalization of
$V_i$ are all rationally chain connected.
Further, $-K_X$ is an ample $\BQQ$-Cartier divisor, i.e., $r > 0$ in $(1)$.
\end{theorem}

A morphism $f : X \to X$ is {\it polarized} (by $H$) if
$$f^*H \sim qH$$
for some ample Cartier divisor $H$
and some $q > 0$; then
$$\deg(f) = q^{\dim X} .$$
For instance,
every non-constant endomorphism of a projective
variety $X$ of Picard number $\rho(X) = 1$, is polarized;
an $f$-stable subvariety $X \subset \BPP^n$ for a non-constant endomorphism
$f : \BPP^n \to \BPP^n$, has the restriction $f_{|X} : X \to X$ polarized by the hyperplane;
the multiplication map
$$m_A : A \to A  \,\,\,\, (x \mapsto m x)$$ (with $m \ne 0$)
of an abelian variety $A$ is polarized by any $H = L + (-1)^*L$ with $L$ an ample divisor,
so that $m_A^*H \sim m^2H$.

In Theorems \ref{ThE'} and \ref{ThG},
we give upper bounds for the number of $f^{-1}$-stable prime divisors on
a (not necessarily smooth) projective variety; the bounds are optimal,
and the second possibility in Theorem \ref{ThG}(2) does occur
(cf.~Examples \ref{ex1ThA} and \ref{ex1ThG}).
One may remove Hyp(A)
in Theorem \ref{ThG},
when the Picard number $\rho(X) = 1$, or $X$ is a weak $\BQQ$-Fano variety,
or the closed cone $\overline{\NE}(X)$ of effective curves has
only finitely many extremal rays
(cf.~Remark \ref{rThB}).
Denote by
$$N^1(X) := \NS(X) \otimes_{\BZZ} \BRR$$
the {\it N\'eron-Severi group}
(over $\BRR$) with $\rho(X) : = \rank_{\BRR} N^1(X)$ the {\it Picard number}.

\begin{theorem}\label{ThG}
Let $X$ be a projective variety of dimension $n$ with
only $\BQQ$-factorial Kawamata log terminal singularities,
and $f : X \to X$ a polarized endomorphism with $\deg(f) = q^n > 1$.
Assume $Hyp(A)$ : either ${f^*}_{|N^1(X)} = q \, \id_{N^1(X)}$, or $n \le 3$.
Then we have $($with $\rho := \rho(X))$:
\begin{itemize}
\item[(1)]
Let $V_i \subset X$ $(1 \le i \le c)$ be prime divisors with $f^{-1}(V_i) = V_i$.
Then $c \le n + \rho$. Further, if $c \ge 1$,
then the pair $(X, \, \sum V_i)$ is log canonical and $X$ is uniruled.
\item[(2)]
Suppose
$c \ge n+\rho-2$. Then either
$X$ is rationally connected,
or there is a fibration $X \to E$ onto an elliptic curve $E$
so that every fibre is normal and rationally connected and
some positive power $f^k$ descends to an $f_E : E \to E$
of degree $q$.
\item[(3)]
Suppose that
$c \ge n+\rho-1$. Then $X$ is rationally connected.
\item[(4)]
Suppose that
$c \ge n+\rho$. Then $c = n+\rho$, $($for some $t > 0)$
$$K_X + \sum_{i=1}^{n+\rho} V_i \, \sim_{\BQQ} \, 0, \hskip 2pc
{(f^t)^*}_{| \Pic(X)} = q^t \, \id_{| \Pic(X)} \, ,$$
$f$ is \'etale outside $\, (\cup \, V_i) \, \cup \, f^{-1}(\Sing X)$
$($and $X$ is a toric surface with $\sum V_i$ its boundary divisor,
when $\dim X = 2)$.
\end{itemize}
\end{theorem}

Corollary \ref{ThB} below
is a special case of Theorem \ref{ThB'} and is known for
$X = \BPP^n$ with $n \le 3$ (cf.~\cite{FS}, \cite{nz2});
the smoothness and Picard number one assumption on $X$
are necessary (cf.~Remark \ref{rThB} and Example \ref{ex1ThG}).
In Corollary \ref{ThB}, $X$ is
indeed a Fano manifold; but one would like to know more about the $V$
and even expects $X = \PP^n$ and $V$ be a hyperplane;
see \cite{p3} and the references therein.
Such an expectation is very hard to prove
even in dimension three and proving the smoothness of $V$ is the key,
hence the relevance of Proposition \ref{ThC} below.

\begin{corollary}\label{ThB}
Let $X$ be a projective manifold of dimension $n \ge 2$ and Picard number one,
$f : X \to X$ an endomorphism of degree $> 1$,
and $\, V \subset X$ a prime divisor with $f^{-1}(V) = V$.
Then $X$, $V$ and the normalization $V'$ of $V$ are
all
rationally chain connected.
\end{corollary}

\begin{corollary}\label{CorD}
With the notation and assumptions in Corollary $\ref{ThB}$,
both $X$ and $V$ are simply connected, while $V'$ has a finite
$($topological$)$ fundamental group.
\end{corollary}

\begin{setup}
{\bf Main ingredients of the proofs}.
The results of Favre \cite{Fa}, Nakayama \cite{ENS} and Wahl \cite{Wah}
are very inspiring about the restriction of the singularity type
of a normal surface
imposed by the existence of an endomorphism of degree $> 1$ on the surface.
For the proof of our results, the basic ingredients are: a log canonical singularity
criterion (Proposition \ref{ThC}), the inversion of adjunction in Kawakita \cite{Kk},
a rational connectedness criterion of Qi Zhang
\cite{Zq} and its generalization in Hacon-McKernan \cite{HM},
the characterization in Mori \cite{Mo} on hypersurfaces in weighted projective spaces, and
the equivariant Minimal Model Program in our early paper \cite{uniruled}.
\end{setup}

Theorems \ref{ThE'} and \ref{ThG} motivate the question below
(without assuming the Hyp(A) in Theorem \ref{ThG}).
Question \ref{Q} (2)
is {\it Shokurov's conjecture} (cf. \cite[Thm 6.4]{Sho}).

\begin{question}\label{Q}
{\rm
Suppose that a projective $n$-fold ($n \ge 3$)
$X$ has only $\BQQ$-factorial Kawamata log terminal singularities,
$f : X \to X$ a polarized endomorphism of degree $> 1$, and $V_i \subset X$
($1 \le i \le s$) prime divisors
with $f^{-1}(V_i) = V_i$.
\begin{itemize}
\item[(1)]
Is it true that
$s \le n + \rho(X)$?
\item[(2)]
If $s = n + \rho(X)$, is it true that
$X$ is a toric variety with $\sum V_i$ its boundary divisor?
\end{itemize}
}
\end{question}

\begin{remark}\label{rThB}
(1) In Corollary \ref{ThB}, it is necessary to assume that
$\rho(X) = 1$ (cf.~Example \ref{ex1ThG}), and $X$ is smooth
or at least Kawamata log terminal (klt).
Indeed, for every projective cone $Y$ over an elliptic curve
and every section $V \subset Y$ (away from the vertex),
there is an endomorphism $f : Y \to Y$ of $\deg(f) > 1$ and with $f^{-1}(V) = V$
(cf.~\cite[Theorem 7.1.1, or Corollary 5.2.3]{ENS}).
The cone $Y$ has Picard number one and a log canonical singularity at its vertex.
Of course, $V$ is an elliptic curve, and is not rationally chain connected.
By the way,
$Y$ is rationally {\it chain} connected, but is not rationally connected.
Observe that $K_Y + V \sim_{\BQQ} 0$, in connection with the condition (3)
in Theorem \ref{ThB'} which is a stronger version of Corollary \ref{ThB}.

(2) Let $X$ be a projective variety with only klt singularities.
If the closed cone $\overline{\NE}(X)$ of effective curves
has only finitely many extremal rays, then every polarized endomorphism
$f : X \to X$ satisfies
$${f^*}_{ | N^1(X)} = q \, \id_{N^1(X)}, \,\,\, \deg(f) = q^{\dim X}$$
after replacing $f$ by its power,
so that we can apply Theorem \ref{ThG} (cf.~\cite[Lemma 2.1]{nz2}).
For instance, if $X$ or $(X, \Delta)$ is weak $\BQQ$-Fano, i.e., $X$ (resp. $(X, \Delta)$)
has only klt singularities and $-K_X$ (resp. $-(K_X + \Delta)$) is nef and big,
then $\overline{\NE}(X)$ has only finitely many extremal rays.

(3) By Example \ref{ex1ThA},
it is necessary to assume the local factoriality of $X$ or the
Cartier-ness of $V_i$ in Theorem \ref{ThE'},
even when $X$ has only klt singularities.
We remark that
a $\BQQ$-factorial Gorenstein terminal threefold is locally factorial.
\end{remark}

A smooth hypersurface $X$ in $\BPP^{n+1}$ with $\deg(X) \ge 3$ and $n \ge 2$,
has no endomorphism $f_X : X \to X$ of degree $> 1$
(cf. \cite{Be}, \cite{CL}).
However, singular $X$ may have plenty of endomorphisms $f_X$ of arbitrary degrees
as shown in Example \ref{ex1ThA} below.

\begin{example}\label{ex1ThA}
{\rm
We now construct many polarized endomorphisms for some degree $n+1$ singular hypersurface
$X \subset \BPP^{n+1}$.
Let
$$f = (F_0, \dots, F_n) : \BPP^n \to \BPP^n \,\, (n \ge 2)$$
with $F_i = F_i(X_0, \dots, X_n)$ homogeneous,
be any endomorphism of degree $q^n > 1$,
such that $f^{-1}(S) = S$ for a reduced degree $n+1$ hypersurface
$S = \{S(X_0, \dots, X_n) = 0\}$.
So $S$ must be normal crossing and linear: $S = \sum_{i=0}^{n} S_i$
(cf.~\cite[Thm 1.5 in arXiv version \bf{1}]{nz2}).
Thus we may assume that
$f = (X_0^q, \dots, X_n^q)$ and
$S_i = \{X_i = 0\} $.
The relation $S \sim (n+1)H$ with $H \subset \BPP^n$ a hyperplane,
defines
$$\pi : X = Spec \oplus_{i=0}^n \SO(-iH) \to \BPP^n$$
which is
a Galois $\BZZ/(n+1)$-cover
branched over $S$ so that $\pi^*S_i = (n+1)T_i$ with the restriction
$\pi_{|T_i} : T_i \to S_i$ an isomorphism.

This $X$ is identifiable with the degree $n+1$ hypersurface
$$\{Z^{n+1} = S(X_0, \dots, X_n)\} \, \subset \, \BPP^{n+1}$$
and has singularity of type $z^{n+1} = xy$
over the intersection points of $S$ locally defined as $xy = 0$.
We may assume that $f^*S(X_0,\dots, X_n) = S(X_0,\dots, X_n)^q$ after replacing
$S(X_0, \dots, X_n)$
by a scalar multiple, so
$f$ lifts to an endomorphism
$$g = (Z^q, F_0, \dots, F_n)$$
of $\BPP^{n+1}$
(with homogeneous coordinates $[Z, X_0, \dots, X_n]$),
stabilizing $X$, so that $g_X := g_{|X} : X \to X$ is a polarized endomorphism
of $\deg(g_X) = q^n$ (cf.~
\cite[Lemma 2.1]{nz2}). Note that
$g^{-1}(X)$ is the union of $q$ distinct hypersurfaces
$$\{Z^{n+1} = \zeta^i S(X_0, \dots, X_n)\} \, \subset \, \BPP^{n+1}$$
(all isomorphic to $X$), where $\zeta := \exp( 2 \pi \sqrt{-1}/q)$.

This $X$ has only Kawamata log terminal singularities
and $\Pic X = (\Pic \BPP^{n+1})_{|X}$ ($n \ge 2$) is of rank one (using Lefschetz type theorem
\cite[Example 3.1.25]{La} when $n \ge 3$).
We have $f^{-1}(S_i) = S_i$ and $g_X^{-1}(T_i) = T_i$ ($0 \le i \le n$).
Note that $(n+1)T_i = \pi^*S_i$ is Cartier, but $T_i$ is not Cartier;
of course $X \not\isom \PP^n$ (compare with Theorem \ref{ThE'}).

If $n = 2$, the relation $(n+1)(T_1 - T_0) \sim 0$ gives rise to an \'etale-in-codimension-one
$\BZZ/(n+1)$-cover
$$\tau: \BPP^{n} \isom \widetilde{X} \to X$$
so that $\sum_{i=0}^{n} \tau^{-1}T_i$
is a union of $n+1$ normal crossing hyperplanes;
indeed, $\tau$ restricted over $X \setminus \Sing X$, is its universal cover
(cf.~\cite[Lemma 6]{MZ}),
so that $g_X$ lifts up to $\widetilde{X}$.
A similar result {\it seems} to be true for $n \ge 3$,
by considering the `composite' of the $\BZZ/(n+1)$-covers given by $(n+1)(T_i - T_0) \sim 0$
($1 \le i < n$); see Question \ref{Q}.
}
\end{example}

The simple Example \ref{ex1ThG} below shows that the conditions in Theorem \ref{ThG} (2)(3), or
the condition $\rho(X) = 1$ in Corollary \ref{ThB}, is necessary.

\begin{example}\label{ex1ThG}
{\rm
Let $m_A : A \to A$ ($x \mapsto mx$) with $m \ge 2$, be the multiplication map of an abelian
variety $A$ of dimension $\ge 1$ and Picard number one.
Let $v \ge 1$, $q := m^2$ and
$$g: \BPP^{v} \to \BPP^{v} \,\,\,\, ([X_1, \dots, X_{v+1}] \mapsto [X_1^q, \dots, X_{v+1}^q]) .$$
Then
$$f = (m_A \times g) : X = A \times \BPP^{v} \to X$$
is a polarized endomorphism
with ${f^*}_{| N^1(X)} = \diag[q, \, q]$,
and $f^{-1}$ stabilizes $v+1$ prime divisors $V_i = A \times \{X_i = 0\} \subset X$
and no others; indeed, $f$ is \'etale outside $\cup \, V_i$.
Note that $X$ and $V_i \isom A \times \BPP^{v-1}$
are not rationally chain connected, and $$v+1 = \dim X + \rho(X) - (1 + \dim A) .$$
}
\end{example}

\noindent
{\bf Acknowledgement.} I would like to thank
N. Nakayama for the comments and informing me about Shokurov's conjecture
(cf. \ref{Q}) and Wahl's result \cite[Corollary, page 626]{Wah},
and the referee for very careful reading and valuable suggestions to improve the paper.

\section{Proofs of Theorems \ref{ThE'} $\sim$ \ref{ThG}}

We use the standard notation in Hartshorne's book and \cite{KM} or \cite{KMM}.
For a finite morphism $f : X \to Y$ between normal varieties (especially for
a surjective endomorphism $f : X \to X$ of a normal projective variety $X$),
we can define
the pullback $f^*L$ on $X$ of a Weil divisor $L$ on $Y$, as the Zariski-closure of $(f_{|U})^*(L_{|V})$
where $U \subset X$ (resp. $V \subset Y$) is a smooth Zariski-open subset
of codimension $\ge 2$ in $X$ (resp. $Y$). When $L$ is $\Q$-Cartier,
our $f^*L$ coincides with the usual pullback (or total transform) of $L$.

In \S 2, we shall prove $\ref{ThE'} \sim \ref{CorD}$
in the Introduction, and Propositions \ref{ThC}, \ref{ThC'} and \ref{ThF} below.

The following log canonical singularity criterion is
frequently used in proving the main results
and should be of interest in its own right.

\begin{proposition}\label{ThC}
Let $X$ be a normal $($algebraic or analytic$)$ variety, $f: X \to X$ a surjective endomorphism
of $\deg(f) > 1$ and $(0 \ne )$ $D$ a reduced divisor with
$f^{-1}(D) = D$. Assume:
\begin{itemize}
\item[(1)]
$X$ is log canonical around $D$ {\rm(cf.}~\cite[Definition~2.34]{KM}$)$;
\item[(2)]
$D$ is $\BQQ$-Cartier; and
\item[(3)]
$f$ is ramified around $D$.
\end{itemize}
Then the pair $(X, D)$ is log canonical around $D$.
In particular, $D$ is normal crossing outside the
union of $\Sing X$ and a codimension three subset of $X$.
\end{proposition}

Proposition \ref{ThC'} below is used in the proof of Theorem \ref{ThE'}.
When $\dim X = 2$, Propositions \ref{ThC} and \ref{ThC'} are shown by
Nakayama \cite[Theorem 4.3.1]{ENS}, and the proof of \cite[Lemma 2.7.9]{ENS}
which seems to be effective in higher dimensions, as commented by Nakayama
(cf. also Wahl \cite[page 626]{Wah} and
Favre \cite[Theorem B]{Fa}); but the {\it log canonical modification} used in \cite[Theorem 4.3.1]{ENS}
was not available then in higher dimensions. To avoid such problem,
in our proof of Propositions \ref{ThC} and \ref{ThC'},
we compute log canonical threshold and discrepancy in the spirit of \cite[Proposition 5.20]{KM}.

\begin{proposition}\label{ThC'}
Let $X$ be a normal $($algebraic or analytic$)$ variety, $f: X \to X$ a surjective endomorphism
of $\deg(f) > 1$ and $(0 \ne )$ $D$ a reduced divisor with
$f^{-1}(D) = D$. Assume:
\begin{itemize}
\item[(1)]
There are effective $\BQQ$-divisors $G$ and $\Delta$ such that
the pair $(X, G)$ has only log canonical singularities around $D$,
and $K_X + G + D = f^*(K_X + G + D) + \Delta$,
i.e., the ramification divisor $R_f = f^*(G+D) - (G+D) + \Delta$;
\item[(2)]
$D$ is $\BQQ$-Cartier; and
\item[(3)]
$f$ is ramified around $D$.
\end{itemize}
Then the pair $(X, G)$ has only
purely log terminal singularities around $D$ \text{\rm (cf. \cite[Def 2.34]{KM})}.
In particular, the structure sheaves
$\OO_X$ and $\OO_D$
are Cohen-Macaulay around $D$.
\end{proposition}

\begin{setup}\label{set2.0}
We now prove Propositions \ref{ThC} and \ref{ThC'}.
We prove Proposition \ref{ThC} first.
Since the result is local in nature, we may assume that $X$ is log canonical.
Consider the following {\it log canonical threshold} of $(X, D)$:
$$c := \max\{t \in \BRR \, | \, (X, tD) \,\,\text{\rm is log canonical}\} .$$
Then $0 \le c \le 1$. We may assume that $c < 1$ and shall reach a contradiction late.
Let $D = \sum D_i$ be the irreducible decomposition with $D_i \ne D_j$ when $i \ne j$.
Since $f^{-1}(D) = D$, we may assume that $f^{-1}(D_i) = D_i$ after
replacing $f$ by its power.
Since $f$ is ramified around $D$,
we may write $f^*D_i = q_iD_i$ for some $q_i > 1$.
Thus $$K_X + D = f^*(K_X + D) + \Delta$$ where $\Delta$ is an effective
integral Weil divisor not containing any $D_i$, so that
$$R_f := \sum (q_i-1)D_i + \Delta$$
is the ramification divisor of $f$.
We can write
$$K_X + cD - \Delta - (1-c) \sum (q_i-1)D_i = f^*(K_X + cD).$$
To distinguish the source and target of $f$, we denote
by $f : X_1 = X \to X_2 = X$.
For the pairs $(X_i, \Gamma_i)$ with
$$X_i := X, \hskip 1pc \Gamma_1 := cD - \Delta - (1-c) \sum (q_i-1)D_i$$
(which is $\BQQ$-Cartier because so are $K_X$ and $f^*(K_X + cD)$)
and $\Gamma_2 := cD$, we apply \cite[Proposition 5.20]{KM}.
By the definition of the log canonical threshold $c$, there is an exceptional divisor
$E_2$ (in a blowup of $X_2$) with its image (the centre) contained in $D \subset X_2$
such that the discrepancy
$a(E_2, X_2, \Gamma_2) = -1$.
Let $E_1$ be an exceptional divisor (in a blowup of $X_1$) which dominates $E_2$,
via a lifting $f'$ of $f$, and hence has image (on $X_1$) contained in $D$.
Here we use the assumption that $f^{-1}(D) = D$.
On the one hand, \cite[Proposition 5.20]{KM} shows that
$$Eq(\ref{ThC}.1) \hskip 2pc a(E_1, X_1, \Gamma_1) + 1 = r(a(E_2, X_2, \Gamma_2) + 1) = 0$$
where $r \ge 1$ is the ramification index of $f'$ along $E_1$.
On the other hand,
by \cite[Lemma 2.27]{KM} and noting that $E_1$
has image in $D$ and hence in the support of the {\it effective} divisor
$$\Delta + (1-c) \sum (q_i-1)D_i = cD - \Gamma_1$$
(which is $\BQQ$-Cartier because so is $\Gamma_1$ as mentioned early on),
we have
$$a(E_1, X_1, \Gamma_1) > a(E_1, X_1, cD) \ge -1$$
since $(X, cD)$ is log canonical. This contradicts the display Eq(\ref{ThC}.1) above.
Therefore, $c \ge 1$ and $(X, D)$ is log canonical.
This proves Proposition \ref{ThC}.

For Proposition \ref{ThC'}, consider, in the notation above,
$$K_X + G - \Delta - \sum (q_i-1) D_i = f^*(K_X + G)$$
and pairs $(X, G - \Delta - \sum (q_i-1) D_i)$ and $(X, G)$.
Then, using \cite[Proposition 5.20, Lemma 2.27, Corollary 5.25]{KM}, Proposition \ref{ThC'}
can be proved as above.
\end{setup}

\begin{setup}\label{set2.1}
We prove Theorem \ref{ThB'}.
By the assumption, $f^{-1}$ stabilizes the reduced divisor $V = \sum_i V_i$,
$-K_X \sim_{\Q} rH$ and $V_i \sim_{\Q} d_i H$ for some $r, \, d_i, \, \in \, \Q$ and
an ample Cartier divisor $H$.
Replacing $f$ by its power, we may assume $f^{-1}(V_i) = V_i$ so that
$f^*V_i = q_iV_i$ for some $q_i > 0$. Since $K_X$ and $V_i$'s are all proportional to $H$,
all $q_i$'s are the same, $f^*K_X \sim_{\BQQ} qK_X$ and $f^*H \sim_{\BQQ} qH$ with
$q := q_i = \root{n}\of{\deg(f)} > 1$.
Write
$$K_X = f^*K_X + R_f$$
with $R_f$ the (effective) ramification divisor.
Then
$R_f = (q-1)V+ \Delta$
with $\Delta$ an effective Weil divisor which does not contain any $V_i$.
Thus
$$K_X + V = f^*(K_X + V) + \Delta$$
and
$$0 \le \Delta \sim_{\BQQ} (1-q)(K_X + V) \sim_{\BQQ} (q-1)(r-d)H$$
where $d := \sum d_i$.
So $r \ge d > 0$.
Let
$$\sigma: V_1' \to V_1$$
be the normalization.
By the subadjuction (cf.~\cite[Corollary 16.7]{Ko+}), we have
$$K_{V_1'} + C' = \sigma^*(K_X + V)_{|V_1} \sim_{\BQQ} -(r -d)\sigma^*(H_{|V_1})$$
where $C'$ is the sum of $\sigma^*(V-V_1)_{|V_1}$, some non-negative contribution
from the singularity of the pair $(X, V)$, and
the conductor of $V_1'$ over $V_1$
(an integral effective Weil divisor).
We set $\sigma = \id$ when $V_1$ is normal.

We apply Proposition \ref{ThC} to $(X, D, f) = (X, V, f)$.
Since $f^*V = qV$ with $q > 1$, our
$f$ is ramified along $V$ with ramification index $q$.
Thus all conditions of Proposition \ref{ThC} are satisfied; hence $(X, D)$
is log canonical around $V$.
By \cite[Theorem]{Kk}, the pair $(V_1', C')$ is also log canonical.
If $\Delta > 0$, i.e., $r > d$ ($> 0$), then
both $-K_X \sim_{\BQQ} rH$ and
$$-(K_{V_1'} + C') \sim_{\BQQ} (r-d) \sigma^*(H_{|V_1})$$
are ample,
so $X$, $V_1'$ (and hence $V_1$) are all rationally chain connected
by \cite[Cor 1.3]{HM}.

Suppose $\Delta = 0$, i.e., $r = d$ ($> 0$).
Then $K_X + V \sim_{\BQQ} 0$, and
$f$ is \'etale outside $V \cup f^{-1}(\Sing X)$
since the ramification divisor
$R_f = (q-1)V$ now and by the purity of branch loci.
If $V = V_1$ then $X$ is smooth by the assumption, so
$X = \BPP^n$ and $V$ is a union of $n+1$ hyperplanes
(cf.~\cite[Theorem 2.1]{HN}), contradicting the irreducibility of $V$.
If $V \ge V_1 + V_2$, then $C'' := C' - \sigma^*({V_2}_{|V_1}) \ge 0$,
the pair $(V_1', C'')$ is log canonical (cf.~\cite[Corollary 2.35]{KM})
and $-(K_{V_1'} + C'') \sim_{\BQQ} \sigma^*({V_2}_{|V_1})$ is ample, so
$V_1'$ is rationally chain connected by \cite[Corollary 1.3]{HM}.
This proves Theorem \ref{ThB'}.
\end{setup}

\begin{setup}
We prove Corollary \ref{CorD}.
By a well known result of Campana \cite{Cp}, a rationally chain connected normal projective
variety $Y$ has a finite (topological) fundamental group $\pi_1(Y)$;
further, $\pi_1(Y) = (1)$ for smooth $Y$.
This, Lefschetz hyperplane section theorem
\cite[Theorem 3.1.21]{La} and Theorem \ref{ThB'}
imply Corollary \ref{CorD} except the triviality of $\pi_1(V)$
when $\dim X = 2$. Now assume $\dim X = 2$. Since $X$ is smooth and rationally chain
connected, $X$ is rational. Thus $X \isom$ $\BPP^2$ since
$X$ has Picard number one. Hence $V$, being $f^{-1}$-stabilized,
is a line (cf. e.g. \cite[Thm 1.5 in arXiv version \bf{1}]{nz2}).
So $V \isom \BPP^1$ is simply connected. This proves Corollary \ref{CorD}.
\end{setup}

\begin{setup}
The results below are
used in the proof of Theorem \ref{ThE'} and Proposition \ref{ThF}.

Let us now define {\it numerical equivalence} on a normal projective surface $S$.
First, one can define intersection form on $S$,
using Mumford pullback. To be precise, let $\tau: S' \to S$ be a minimal resolution.
For a Weil divisor $D$ on $S$, define the pullback $\tau^*D := \tau'D + \sum a_i E_i$
where $\tau'D$ is the proper transform of $D$ and $E_i$ are $\tau$-exceptional curves, and
$a_i \in \R$ are uniquely determined (by the negativity of the matrix $(E_i . E_j)$ and)
the condition $\tau^* D . E_j = 0$ for all $j$.
We define the intersection $D_1 . D_2 := \tau^*D_1 . \tau^* D_2$.
Weil divisors $D_1$ and $D_2$ on $S$ are called numerically equivalent
if $D_1 . C = D_2 . C$ for every curve $C$ on $S$.
This way, we have defined an equivalence relation among Weil divisors on $S$.
The equivalence class containing $D$ is called the {\it numerical Weil divisor class}
containing $D$.
\end{setup}

Lemma \ref{pe} below is known to Iitaka, Sommese, Y. Fujimoto, and Nakayama \cite[Lemma 3.7.1]{ENS}, \dots.
We reprove it here for the convenience of the readers.

\begin{lemma}\label{pe}
Let $X$ be a normal projective variety of dimension $n$ and $f : X \to X$ an endomorphism with $\deg(f) \ge 2$.
Supposer that the canonical $($Weil$)$ divisor $K_X$ is pseudo-effective $($see \cite[Ch II, Definition 5.5]{ZDA}$)$.
Then
$f$ is \'etale in codimension one.
\end{lemma}

\begin{proof}
Write $K_X = f^*K_X + R_f$ with $R_f \ge 0$ the ramification (integral) divisor,
noting that the pullback $f^*$ is defined at the beginning of \S 2.
Substituting this expression of $K_X$ to the right hand side $(s-1)$-times, we get
$K_X = (f^s)^*K_X + \sum_{i=0}^{s-1} (f^i)^*R_f$.
Take an ample Cartier divisor $H$ on $X$.
If $R_f = 0$, then we are done.
Otherwise,
the pseudo-effectivity of $K_X$ and \cite{BDPP} imply that $(R_f$ being an integral Weil divisor)
$$K_X . H^{n-1} = (f^s)^*K_X . H^{n-1} + \sum_{i=0}^{s-1} (f^i)^*R_f . H^{n-1}
\ge s .$$
Let $s \to \infty$. We get a contradiction.
\end{proof}

\begin{lemma}\label{surf}
Let $S$ be a normal projective surface,
$\tau : S' \to S$ the minimal resolution,
$f: S \to S$ a polarized endomorphism
with $\deg(f) = q^2 > 1$, and $D$, $\Delta$ effective Weil divisors such that
$K_S + D = f^*(K_S + D) + \Delta$, i.e., the ramification divisor $R_f = f^*D - D + \Delta$.
Suppose that $D$ is an integral divisor, and
$\Supp D = \cup_{i=1}^r D_i$ has $r \ge 3$ irreducible components
and $($the dual graph of$)$ it contains a loop. Replacing $f$ by its power, we have:
\begin{itemize}
\item[(1)]
$S$ is klt. The pair $(S, D)$ has only log canonical singularities;
so no three of $D_i$ share the same point.
\item[(2)]
$D$ is reduced, and $f^*D_i = qD_i$ for every $i$.
\item[(3)]
$\Delta = 0$,
$K_S + D \sim 0$, and $f$ is \'etale outside $D \cup f^{-1}(\Sing S)$.
\item[(4)]
$S$ is a rational surface.
Every singularity of $S$ is either Du Val and away from $D$, or is
a cyclic quotient singularity and lies in $\Sing D$.
$\tau^{-1}D$ is a simple loop of $\BPP^1$'s.
\item[(5)]
$f^*L \sim qL$ $($resp.~$f^*L \sim_{\BQQ} qL)$ for every Cartier $($resp.~Weil$)$
divisor $L$ on $S$, so
$f^* = q \, \id$ on $\Pic S$ $($resp. on Weil divisor classes$)$.
\end{itemize}
\end{lemma}

\begin{proof}
For (1), by \cite[Theorem 4.3.1]{ENS} or \cite[Theorem B]{Fa}, both $S$ and the pair $(S, D)$
have only log canonical singularities. We will see that $S$ is klt in (4).

For (2), see \cite[Lemmas 5.3 and 2.1 in arXiv version \bf{1}]{nz2}.

By (2), $f$ is not \'etale in codimension one, and hence
$K_S$ is not pseudo-effective (cf. Lemma \ref{pe} or \cite[Lemma 3.7.1]{ENS}).  So $S'$ is a ruled surface.
Also, $f^* = q \id$, on the numerical Weil divisor classes, after $f$ is replaced by its power
(cf.~\cite[Theorem 2.7]{uniruled}).
Thus
$$Eq(\ref{surf}.1) \hskip 2pc
0 \le \Delta \equiv -(q-1)(K_S + D), \,\,\,\, -K_S \equiv D + \Delta/(q-1).$$

Since $f^* = q \id$ and $K_S$ is not pseudo-effective,
the classification result of \cite[Theorem 6.3.1]{ENS} says that
$S$ is either a rational surface, or an elliptic (smooth minimal) ruled surface,
or a cone over an elliptic curve.
If $S$ is elliptic ruled with $F$ a general fibre, then intersecting $F$ with Eq(\ref{surf}.1) above
and noting that $D$ contains a loop,
we may assume that $F . (D_1 + D_2) = 2$ and $F . \Delta = F . D_j = 0$ ($j = 3, \dots, r$);
then $K_S + D_1 + D_2$ is pseudo-effective by using Hartshorne's book, Chapter V, Propositions
2.20 and 2.21, which gives a contradiction:
$$K_S + D_1 + D_2 \equiv -(\sum_{i \ge 3} D_i + \Delta/(q-1)) < 0 .$$
If $S$ is a cone then $K_S + D_1$ is pseudo-effective,
since $D$ contains a loop and hence we can find
some $D_1 \le D$ horizontal to generating lines, a contradiction as above.

Thus $S$ is a rational surface.
Write
$$\tau^*(K_S + D) = K_{S'} +  D' + \Sigma_1 + \Sigma_2 + \Sigma_3$$
where $D' = \tau'D$ is the proper transform of $D$,
$\Sigma_i \ge 0$,
$$\begin{aligned}
\Supp \Sigma_1 &= \tau^{-1}((\Sing S) \cap (\Sing D)), \\
\Supp \Sigma_2 &= \tau^{-1}((\Sing S) \cap (D \setminus \Sing D)), \\
\Supp \Sigma_3 &\subset \tau^{-1}((\Sing S) \setminus D) .
\end{aligned}$$
By the results on (1) in the first paragraph and \cite[Theorem 9.6]{Ka},
($D$ and hence) $D' + \Sigma_1$
are reduced and contain a loop. Thus
$K_{S'} + D' + \Sigma_1 \sim G \ge 0$ by the Riemann-Roch theorem
(cf. \cite[Lemma 2.3]{CCZ}).
Pushing forward, we get $K_S + D \sim \tau_*G \ge 0$.
This and the displayed Eq(\ref{surf}.1) above imply $\Delta = 0 = \tau_*G$,
so (3) is true by the assumption on $R_f$ and the purity of branch loci.

Since $$0 \sim \tau^*(K_S + D) = K_{S'} +  D' + \Sigma_1 + \Sigma_2 + \Sigma_3
\sim G + \Sigma_2 + \Sigma_3$$
we have $G = 0 = \Sigma_i$ ($i = 2, 3$).
Now (4) follows from $\Sigma_i = 0$ ($i = 2, 3$),
the results on (1) in the first paragraph,
and the Riemann-Roch theorem (cf. the proofs of \cite[Lemmas 2.2 and 2.3]{CCZ}).

(5) follows from \cite[Theorem 2.7]{uniruled}. Indeed,
$S$ is klt and hence $\BQQ$-factorial.
The argument below is valid in any dimension for later use:
since $S$ is klt (and rational, i.e. rational connected), $S'$ is rational (i.e.,
rational connected, cf. \cite[Corollary 1.5]{HM});
thus $\pi_1(S')$ (and hence $\pi_1(S)$) are trivial (hence $q(S) = 0$),
by a well-known result of Campana \cite{Cp};
so $\Pic S$ is torsion free.
\end{proof}

\begin{setup}
{\bf Proof of Theorem \ref{ThE'}}
\end{setup}

By the assumption, $f^{-1}(V_i) = V_i$ and $V_i \equiv d_iH$ for some $d_i > 0$,
so each $V_i$ is an ample Cartier divisor.
Suppose there are $s \ge n+1$ of such $V_i$.
We have $f^*V_i = q V_i$ since
$q^n = \deg(f) = (f^*V_i)^n/V_i^n$.
So $f$ is polarized by $V_1$. We may assume that $H = V_1$,
since all $V_i$ are (numerically) proportional to each other by the assumption.
We shall inductively construct
log canonical pairs $(X_i, D_i)$ ($1 \le i \le n-2$) with $\dim X_i = n-i$.
Let
$$X_0 := X, \hskip 1pc D_0 := \sum_{i=1}^s V_i .$$
By Proposition \ref{ThC}, the pair $(X_0, D_0)$ is log canonical around $D_0$.
Let $\sigma_1 : X_1 \to X_0$ be the normalization of $V_1 \subset X_0$.
Write
$$K_{X_1} + D_1 = \sigma_1^*(K_{X_0} + D_0)$$
so that the pair $(X_1, D_1)$ is again log canonical (cf.~\cite[Theorem]{Kk}).
Let $\Gamma_1 \subset X_1$ be the conductor divisor of $\sigma_1$.
Set $\Gamma_1 = 0$ when $V_1$ is normal.
By the calculation of the {\it different} in \cite[Corollary 16.7]{Ko+},
$$D_1 \ge \Gamma_1 + \sum_{k=2}^s \sigma_1^*V_k ;$$
hence the right hand side is reduced and each of its last $s-1$ term is nonzero
and connected by the
ampleness of $V_k$ ($\equiv d_kH$).
Our $f$ lifts to an endomorphism $f_1 : X_1 \to X_1$
polarized by $\sigma_1^*H$ so that $f_1^{-1}$ stabilizes $\Gamma_1$
and $\sigma_1^*V_k$ ($k \ge 2$)
after replacing $f$ by its power (cf.~\cite[Proposition 5.4 in arXiv version \bf{1}]{nz2}).
We repeat the process.
Let $\sigma_2 : X_2 \to X_1$ be the normalization of an irreducible component
of $\sigma_1^*V_2 \subset X_1$
which meets $\Gamma_1$
when it is nonzero; here we use the ampleness of $V_2$ ($\equiv d_2H$).
Write
$$K_{X_2} + D_2 = \sigma_2^*(K_{X_1} + D_1)$$
so that the pair $(X_2, D_2)$ is again log canonical.
We have
$$D_2 \ge \sigma_2^* \Gamma_1 + \sum_{k=3}^s \sigma_2^*\sigma_1^*V_k .$$
Our $f_1$ lifts to an endomorphism $f_2 : X_2 \to X_2$
polarized by $\sigma_2^*\sigma_1^*H$ so that $f_2^{-1}$ stabilizes
each term of $\sigma_2^* \Gamma_1 + \sum_{k=3}^s \sigma_2^*\sigma_1^*V_k$
after replacing $f$ by its power.
Thus we can construct normalizations (onto the images)
$$\sigma_i : X_i \to X_{i-1} \hskip 1pc (1 \le i \le n-2) ,$$
log canonical pairs $(X_i, D_i)$ with
$$D_i \ge (\sigma_2 \cdots \sigma_{i})^*\Gamma_1 + \sum_{k = i+1}^s (\sigma_1 \cdots \sigma_{i})^*V_k$$
and endomorphisms $f_i : X_i \to X_i$ polarized by the pullback of $H$ and hence
of $\deg(f_i) = q^{\dim X_i} = q^{n-i}$ (cf.~\cite[Lemma 2.1]{nz2}).

$S := X_{n-2}$ is a normal surface with ample reduced (Cartier) divisors
$$C_i := (\sigma_1 \cdots \sigma_{n-2})^*V_i \hskip 1pc
(n-1 \le i \le s)$$ so that $(S, D_{n-2})$,
$(S, C)$ and $S$ are all log canonical,
where $$C := \sum_{i=n-1}^s C_i \le D_{n-2}$$
(cf.~\cite[Notation 4.1]{KM}, or \cite[Remark 2.7.3, Theorem 2.7.4]{ENS}).
By the construction, $f_{n-2}^{-1}$ stabilizes $C_i$ and hence its irreducible components
$C_{ij}$ (after replacing $f$ by its power), so $f_{n-2}^*C_{ij} = qC_{ij}$
(cf.~\cite[Lemma 2.1]{nz2}).
Write
$$K_S + C = f_{n-2}^*(K_S + C) + \Delta$$ with
an effective Weil divisor $\Delta$ containing no any $C_{ij}$.

\begin{claim}\label{c2.1}
The following are true.
\begin{itemize}
\item[(1)]
$s = n+1$, $K_S + C \sim 0$, $\Delta = 0$, and $f_{n-2} : S \to S$ is
\'etale outside $C \cup f_{n-2}^{-1}(\Sing S)$.
\item[(2)]
$S = X_{n-2} \isom \BPP^2$, and
$C = \sum C_i = \sum_{i=n-1}^{n+1} (\sigma_1 \cdots \sigma_{n-2})^* V_i$
is the sum of three normal crossing lines.
\end{itemize}
\end{claim}

\begin{proof}
Since each $C_i$ is ample, $\sum_{i=n-1}^{n+1} C_i$ contains a loop.
Then (1) follows from Lemma \ref{surf}, noting that $f_{n-2}^*C_i = qC_i$
($i = n-1, \dots, s)$ with $s \ge n+1$.
Since $-K_S \sim C$, our $S$ is Gorenstein and also klt by Lemma \ref{surf},
so $S$ is a del Pezzo surface with only Du Val singularities.
Since all three $C_i$ in $C$ are ample Cartier divisors, we have $K_S^2 = C^2 \ge 9$.
Then it is known that $S \isom \BPP^2$.

Indeed, let $\tau: S' \to S$ be the minimal resolution. Then $K_{S'} = \tau^*K_S$
since $S$ has only Du Val singularities. Thus $S'$ is a smooth rational surface
with $K_{S'}^2 = K_S^2 \ge 9$. Let $S' \to S_m$ be the blowdown to a relatively
minimal smooth rational surface.
Then $S_m$ is either $\BPP^2$ with $K_{S_m}^2 = 9$ or a Hirzebruch surface
with $K_{S_m}^2 = 8$; and we have
$K_{S_m}^2 \ge K_{S'}^2$ where equality holds only when $S'= S_m$.
Thus $S' = S_m = \BPP^2$. Hence the contraction
$\tau: S' \to S$ is an isomorphism because there is no curve on $\BPP^2$ that can be possibly
contracted to a point on $S$.

Since $f_{n-2}^{-1}$ stabilizes $C_i$ and its irreducible
components (after replacing $f$ by its power), our (2) follows from
\cite[Thm 1.5 in arXiv version \bf{1}]{nz2}.
\end{proof}

\begin{claim}\label{c2.2}
Every $\sigma_{k+1}$ $(0 \le k \le n-3)$ is an embedding onto its image,
i.e., $(\sigma_1 \cdots \sigma_{k})^*$ $V_{k+1}$ is an $($irreducible$)$
normal Cohen-Macaulay variety.
Hence every $V_i$ $(1 \le i \le s = n+1)$ is a normal variety.
\end{claim}

\begin{proof}
Since $(X_{k}, D_{k})$ is log canonical and $D_{k} = (\sigma_1 \dots \sigma_{k})^*V_{k+1} +$
(other effective divisor), $(X_{k}, D_{k} - (\sigma_1 \dots \sigma_{k})^*V_{k+1})$
is also log canonical (cf.~\cite[Corollary 2.35]{KM}).
By Proposition \ref{ThC'} (for its condition (1) about $R_{f_{k}}$,
see the proof of Proposition \ref{ThF} below), the reduced divisor
$(\sigma_1 \dots \sigma_{k})^*V_{k+1}$ is Cohen-Macaulay. If this reduced divisor
is not regular in codimension one, then the conductor divisor of $X_{k+1}$ over it
will give rise to an effective divisor $\Theta$ in $D_{n-2} - C$ (as we did for $\Gamma_1$)
which is preserved by $f_{n-2}^{-1}$ so that $f_{n-2}^*\Theta = q\Theta$
(cf.~\cite[Lemma 2.1]{nz2}), contradicting Claim \ref{c2.1}.
Thus, by Serre's $R_1+S_2$ criterion, this reduced divisor is normal
(and also connected by the ampleness of $V_i$).
The second assertion follows from the first with $k = 0$ and by relabeling $V_k$.
\end{proof}

We continue the proof of Theorem \ref{ThE'}.
By Claim \ref{c2.2}, $X_k$ ($1 \le k \le n-2$) is equal to $\cap_{i=1}^k V_i$
and Cohen-Macaulay.
We now apply \cite[Theorem 3.6]{Mo} to show inductively the assertion
that $$(X_i, \SO({V_{i+1}}_{|X_i})) \isom (\BPP^{n-i}, \SO(1)) \hskip 1pc (0 \le i \le n-2).$$
Relabel $V_i$ so that $d_{n+1} \ge d_n \ge \cdots \ge d_1$.
By Claim \ref{c2.1}, $$(X_{n-2}, \SO({V_{n-1}}_{|X_{n-2}})) \isom (\BPP^2, \SO(1)).$$
Note that $X_{n-2} = {V_{n-2}}_{| X_{n-3}}$, and $\SO({X_{n-2}}_{| X_{n-2}}) \isom \SO_{\BPP^2}(1)$
because
$$1 \le ({X_{n-2}}_{|X_{n-2}})^2 = ({V_{n-2}}_{|X_{n-2}})^2 \le ({V_{n-1}}_{|S})^2 = C_{n-1}^2 = 1$$
(here we used $d_{n-1} \ge d_{n-2}$). Thus, by [ibid.],
$$(X_{n-3}, \SO({V_{n-2}}_{|X_{n-3}})) \isom (\BPP^3, \SO(1)).$$
Suppose the assertion is true for $i \ge k$. Then
$$\BPP^{n-k} \isom X_{k} = {V_{k}}_{ | X_{k-1}}, \,\,\,\, \SO({X_{k}}_{ | X_{k}}) \isom \SO_{\BPP^{n-k}}(1)$$
because
$$\begin{aligned}
1 &\le ({X_{k}}_{|X_{k}})^{n-k} = ({V_{k}}_{|X_{k}})^{n-k} \le ({V_{k+1}}_{|X_{k}})({V_{k}}_{|X_{k}})^{n-k-1}
= X_{k+1} ({V_{k}}_{|X_{k}})^{n-k-1} \\
&= ({V_{k}}_{|X_{k+1}})^{n-k-1} \le
({V_{k+1}}_{|X_{k+1}})^{n-k-1} \le
\cdots \le  ({V_{n-2}}_{|X_{n-2}})^2 = 1.
\end{aligned}$$
Thus the assertion is true for $i = k-1$ by [ibid.].
This proves the assertion.

Now take $H \subset X = \BPP^n$ to be the hyperplane
and $d_i = \deg(V_i)$.
We have
$$K_X + \sum V_i = f^*(K_X + \sum V_i) + N$$
where $N$ is an effective Weil divisor.
Thus
$$0 \le N \sim (1-q)(K_X + \sum V_i) \sim (q-1)(n+1 - \sum d_i) H$$
and hence $n+1 \ge \sum_{i=1}^{n+1} d_i \ge n+1$.
So $d_i = 1$. By \cite[Thm 1.5 in arXiv version \bf{1}]{nz2},
$\cup \, V_i$ is a normal crossing union of $n+1$ hyperplanes,
so that we may assume that $V_i = \{X_i = 0\}$,
and also $f^*X_i = X_i^q$ (after replacing $X_i$ by a scalar multiple)
since $f^*V_i = qV_i$.
This proves Theorem \ref{ThE'} because the last `if part' is clear.

\par \vskip 1pc
If the Cartier-ness of $V_i$ in Theorem \ref{ThE'} is replaced by the weaker $\BQQ$-Cartier-ness, we have
the following, where the condition (2) is true when $\rho(X) = 1$ and $X$ is $\BQQ$-factorial.

\begin{proposition}\label{ThF}
Let $X$ be a normal projective variety of dimension $n \ge 2$,
$V_i$ $(1 \le i \le s)$ prime divisors,
and $f : X \to X$ an endomorphism with $\deg(f) = q^n > 1$ such that:
\begin{itemize}
\item[(1)]
$X$ has only log canonical singularities around $\, \cup \, V_i \,$;
\item[(2)]
every $V_i$ is $\BQQ$-Cartier and ample; and
\item[(3)]
$f^{-1}(V_i) = V_i$ for all $i$.
\end{itemize}
Then $s \le n+1$; and $s = n + 1$ only if:
$f$ is \'etale outside
$(\cup \, V_i) \cup f^{-1}(\Sing X)$, and
$\cap_{i=1}^t V_{b_i} \subset X$
is a normal $($irreducible$)$ subvariety for every subset
$\{b_1, \dots, b_t\} \subseteq \{1, \dots, n+1\}$
with $1 \le t \le n-2$ $($and further, $K_X + \sum_{i=1}^{n+1} V_i \equiv 0$
provided that $\rho(X) = 1$ or $f^*K_X \equiv qK_X)$.
\end{proposition}

\begin{proof}
We assume that $s \ge n+1$ and use the notation and steps in the proof of Theorem \ref{ThE'}.
Then $f^*V_i = qV_i$ and $f$ is polarized by a multiple $H$ of $V_1$.
Write $$K_X + D_0 = f^*(K_X + D_0) + \Delta_f$$ with
$\Delta_f$ an effective divisor containing no any $V_i$.
Pulling back by the normalization (onto $V_1$) $\sigma_1 : X_1 \to X_0 = X$,
we have
$$K_{X_1} + D_1 = f_1^*(K_{X_1} + D_1) + \sigma_1^* \Delta_f$$
where $f_1 : X_1 \to X_1$ is lifted from $f_{|V_1}$ and polarized by $\sigma_1^*H$.
By \cite[Lemmas 5.3 and 2.1, proof of Proposition 5.4 in arXiv version \bf{1}]{nz2},
$\sigma_1^* \Delta_f$ contains no any component of $D_1$,
our $D_1$ is reduced, and $f_1^*D_{1j} = qD_{1j}$
for every irreducible component $D_{1j}$ of $D_1$ (after $f$ is replaced by its power).
As in the proof of Theorem \ref{ThE'},
for $1 \le i \le n-2$,
we have log canonical pairs $(X_i, D_i)$
with $D_i = \sum_j D_{ij}$
reduced, and normalizations (onto the images) $\sigma_i : X_i \to X_{i-1}$.

We still have Claims \ref{c2.1}(1) (hence $s = n+1$) and \ref{c2.2}
with the same proof, noting that
$(\sigma_1 \cdots \sigma_t)^*V_{t + v}$ ($v \ge 1$)
are reduced by the argument above for
all pairs $(X, \, \sum_{k=0}^{t} V_k)$.
So $\cap_{k=1}^{i} V_k = X_i$ (a normal variety).
By the construction, inductively, we can write
$$Eq(\ref{ThF}.1) \hskip 2pc K_{X_i} + D_i = f_i^*(K_{X_i} + D_i) + {\Delta_f}_{ | X_i}$$
so that $f_i^*D_{ij} = qD_{ij}$ and $R_{f_i} = (q-1)D_i + {\Delta_f}_{ | X_i}$.
These, together with $D_{n-2} \ge C$ and Claim \ref{c2.1}(1), imply that $D_{n-2} = C$
and ${\Delta_f}_{ | X_{n-2}} = 0$,
so $\Delta_f = 0$ by the ampleness of $V_i$ and hence $R_f = (q-1) D_0$.
For the second assertion, we use the purity of branch loci, Claim \ref{c2.2}, and
the relabeling of $V_k$.

If $f^*K_X \equiv qK_X$ or $\rho(X) = 1$
(and using \cite[Lemma 2.1]{nz2}), the Eq(\ref{ThF}.1) above (with $i = 0$) implies
$K_X + D_0 \equiv q(K_X + D_0)$ and hence the last part of the proposition.
\end{proof}

\begin{setup}
{\bf Proof of Theorem \ref{ThG}}
\end{setup}

By the assumption, $f : X \to X$ is a polarized endomorphism with $\deg(f) = q^n > 1$;
and either $n = \dim X \le 3$, or ${f^*}_{|N^1(X)} = q \, \id_{N^1(X)}$.
We need to prove the four assertions in Theorem \ref{ThG}.
Our proof will be by the induction on $\dim X$.
The case $\dim X = 1$ follows from the Hurwitz formula. Suppose Theorem \ref{ThG} is true
for those $X'$ with $\dim X' \le n-1$. Consider the case $\dim X = n \ge 2$.
We may assume that there are prime divisors
$V_j$ ($1 \le j \le s$) with $f^{-1}(V_j) = V_j$ for some
$s \ge \rho(X) + n - 2 \ge 1$.
We will mainly prove (1) and (4) of Theorem \ref{ThG} because (2) and (3) are similar
and easier. So we may assume that $s \ge n + \rho(X) \ge n+1 \ge 3$.
By the assumption, $f^*H \sim qH$, with an ample Cartier divisor $H$
and $q = \root{n}\of{\deg(f)} > 1$; further, $f^*V_j = qV_j$
(cf.~\cite[Lemma 2.1]{nz2}).
So one may compute the ramification divisor of $f$ as:
$$R_f = (q-1) \sum V_i + \Delta$$
with
$\Delta$ an effective Weil divisor containing no any $V_j$,
and hence
$$K_X + \sum V_j = f^*(K_X + \sum V_j) + \Delta .$$
Further, the second part of (1) follows from Proposition \ref{ThC}.
Indeed, since $R_f > 0$, i.e., $f$ is not \'etale in codimension one,
$K_X$ is not pseudo-effective (cf. Lemma \ref{pe} or \cite[Lemma 3.7.1]{ENS}). Hence $X$ is uniruled
by the well known results of Mori-Mukai and Boucksom-Demailly-Paun-Peternell
(\cite{MM}, \cite{BDPP}).

Let $X \dasharrow X_1$ be a
divisorial extremal
contraction or a flip (and then $n \ge 3$).
Let $V(1)_j \subset X_1$ ($1 \le j \le s_1$) be the image of $V_j$ when $V_j$
is not exceptional over $X_1$. Thus $s_1 \ge n + \rho(X_1)$
since $s_1 \ge s-1$ and $\rho(X_1) = \rho(X) - 1$ (resp.~$s_1 = s$ and $\rho(X_1) = \rho(X)$)
when $X \dasharrow X_1$ is divisorial (resp.~flip).
The map $f$, replaced by its power, descends to a holomorphic
endomorphism $f_1 : X_1 \to X_1$ of degree $q^n$,
by using \cite[Theorem 1.1, Lemmas 3.6 and 3.7]{uniruled}
(under the condition $n = 3$ or ${f^*}_{ | N^1(X)} = q \, \id_{N^1(X)}$)
and \cite[Proposition 3.6.8]{ENS} (saying that negative curves are $f^{-1}$-periodic,
under the condition $n = 2$).
Clearly, $f_1^{-1}V(1)_j = V(1)_{j}$.
Continuing the process, we have a composition
$$X = X_0 \dasharrow X_1 \dasharrow \cdots \dasharrow X_r$$
of divisorial contractions
and flips, holomorphic maps $f_i : X_i \to X_i$ induced from $f$ (replaced by its power),
and prime divisors $V(i)_j \subset X_i$ ($1 \le j \le s_i$) which are the images of $V_j$, for some
$s_i \ge n + \rho(X_i) \ge 3$.
Note that $f_i^*V(i)_j = qV(i)_j$ and hence
$$R_{f_i} \ge (q-1) \sum_j V(i)_j > 0 .$$
Thus, as reasoned above for $X$, our $K_{X_i}$ is not pseudo-effective
and hence $X_i$ is uniruled.
Further, denoting by $\Delta(i)$ the image of $\Delta$, we have
$$(*) \hskip 2pc K_{X_i} + \sum_j V(i)_j = f_i^*(K_{X_i} + \sum_j V(i)_j) + \Delta(i)$$
By the minimal model program and \cite[the proof of Corollary 1.3.2]{BCHM}
and since $K_X$ is not pseudo-effective as mentioned early on,
we may assume that some $W := X_r$ has
an extremal contraction
$$\pi: W = X_r \to Y$$ of fibration type
so that $\dim Y \le n-1$, and
$Y$ and $X_i$ all have only $\BQQ$-factorial
Kawamata log terminal (klt) singularities
(cf. \cite[Lemma 5-1-5, Propositions 5-1-6 and 5-1-11]{KMM}, \cite[Ch VII, Corollary 3.3]{ZDA}).
Further, $f_r$ is polarized
by some ample divisor $H_W$ of degree $q^n$ and it descends to an
endomorphism (polarized by some ample Cartier divisor $H_Y$
of degree $q^{\dim Y}$):
$$f_Y : Y \to Y ;$$
see
\cite[Theorem 2.13, Lemma 2.12; Theorems 1.1 and 2.7 for $n \le 3$]{uniruled}, \cite[Lemma 2.3]{nz2};
and note that
if ${f^*}_{| N^1(X)} = q \, \id_{N^1(X)}$ on $X$ then the same is true on all of $X_i$ and $Y$.
This is the second place we use the Hyp(A) in Theorem \ref{ThG}.

\begin{setup}
{\bf The case \(\dim Y = 0\).} Then $\rho(W) = 1$. Thus $-K_W$ is ample ($W$
being uniruled)
and $W$ and hence $X$ are rationally connected (cf. \cite[Theorem 1]{Zq}, \cite[Corollary 1.5]{HM}),
so $q(X) = q(W) = 0$ (cf. the proof of Lemma \ref{surf}).
Since $s_r \ge n + \rho(X_r) = n+1$, by Proposition \ref{ThF},
$s_r = n + 1$ and $K_W + \sum_{j=1}^{n+1} V(r)_j$ is numerically (and hence $\BQQ$-linearly,
$q(W)$ being zero)
equivalent to zero.
Then, by the construction, $s_i = n + \rho(X_i)$ for all $i$
(so Theorem \ref{ThG}(1) is true)
and the exceptional divisor of $X_i \to X_{i+1}$ is contained
in $\sum_j V(i)_j$ when the map is divisorial. Hence $\sum V_j$ is the sum
of the proper transform of $\sum_j V(r)_j$ and the exceptional divisors
of the composite
$$\delta: X = X_0 \dasharrow X_r = W .$$
Write
$$K_X + \sum_j V_j = \delta^*(K_W + \sum_j V(r)_j) + E_2 - E_1 \sim_{\BQQ} E_2 - E_1$$
for some effective $\delta$-exceptional divisors $E_i$ (with no common components)
whose supports are hence contained in $\cup \, V_j$,
so $f^* E_i = q E_i$; here
the $\delta$-pullback is well defined since $\delta$ involves
only flips and holomorphic maps.
By the display (*) above,
$$E_2 - E_1 - \Delta \sim_{\Q} K_X + \sum_j V_j - \Delta = f^*(K_X + \sum_j V_j) \sim_{\Q} f^*(E_2 - E_1) = q (E_2 - E_1).$$
Thus $\Delta + (q-1) E_2 \sim_{\Q} (q-1) E_1 $.
So $E_1 = E_2 = \Delta = 0$, because $E_1$ is $\delta$-exceptional
and hence has Iitaka $D$-dimension $\kappa(X, E_1) = 0$, and
$\Supp E_1 \subseteq \cup \, V_j$ and $\Supp(E_2 + \Delta)$
have no common components.
Now
$K_X + \sum_j V_j \sim_{\BQQ} 0$,
so $R_f = (q-1) \sum V_j$ and $f$ is \'etale outside $(\cup \, V_j) \cup f^{-1}(\Sing X)$
by the purity of branch loci.

Since $\rho(W) = 1$ (and $q(W) = 0$) and hence $\Pic(X)$ is spanned (over $\BQQ$) by
$H_W$ and $V_j$ with $f^*H_W \sim qH_W$ and $f^*V_j = qV_j$,
we have ${f^*}_{| \Pic X} = q \, \id_{|\Pic X}$ (with $f$ having been replaced by its power),
by the proof of Lemma \ref{surf}.
This proves Theorem \ref{ThG}(4) in the present case;
see \cite[Theorem 6.4]{Sho} for the assertion about $(X, \sum V_i)$ being a toric pair when $\dim X = 2$.
\end{setup}

\begin{setup}
{\bf The case \(1 \le \dim Y \le n-1 \).}
By \cite[Theorem 5.1]{Fr}, the $f_Y$-periodic points are dense and we let
$y_0$ be a general one of them so that $f_Y(y_0) = y_0$ (after replacing $f$ by its power).
Then
$W_0 := \pi^{-1}(y_0)$ is a klt Fano variety.
The restriction $$f_{W_0} = {f_r}_{| W_0} : W_0 \to W_0$$ is an endomorphism of degree $q^{\dim W_0} > 1$
and polarized by ${H_W}_{ | W_0}$.

If the restriction $\pi : V(r)_i \to Y$ ($1 \le i \le s_r(1)$) is not surjective, then ${V(r)_i}_{ | W_0} = 0$
and $V(r)_i$ is perpendicular to any fibre of $\pi : W \to Y$ so that
$$V(r)_i = \pi^*G_i$$
for some prime divisor $G_i \subset Y$ (since the relative Picard number $\rho(W/Y) = 1$)
and also $f_Y^{-1}(G_i) = G_i$; by the inductive hypothesis, $s_r(1) \le \dim Y + \rho(Y)$.

If the restriction
$$\pi : V(r)_j \to Y \,\,\,\, (s_r(1) + 1 \le j \le s_r(1) + s_r(2) = s_r)$$
is surjective,
then these ${V(r)_j}_{ | W_0}$ are ample since $\rho(W/Y) = 1$, and they share no common irreducible
component by the general choice of $W_0 = \pi^{-1}(y_0)$; by the proof of Proposition \ref{ThF},
$s_r(2) \le \dim W_0 + 1$.
Thus
$$n + \rho(X_r) \le s_r \le \dim Y + \rho(Y) + \dim W_0 + 1 = n + \rho(X_r)$$
and all inequalities are actually equalities,
so $s_i = n + \rho(X_i)$ for all $i$ as above, and Theorem \ref{ThG}(1) is true.
Applying the inductive hypothesis on $Y$
we conclude that:
$$(**) \,\,\,\,\,s_r(1) = \dim Y + \rho(Y), \,\,\,\, K_Y + \sum_{i=1}^{s_r(1)} G_i \sim_{\BQQ} 0,
\,\,\,\, {f_Y^*}_{ | \Pic Y} = q \, \id_{\Pic Y} $$
(with $f$ replaced by its power) and that $Y$ is rationally connected, so
$W$ (and hence $X$) are rationally connected by \cite{GHS}.

(For Theorem \ref{ThG} (2) or (3), we have $s_i - (\dim X_i + \rho(X_i) - 2) \ge 0$, or $\ge 1$,
and hence $s_r(1) -(\dim Y + \rho(Y) - 2) \ge 0,$ or $\ge 1$, respectively,
by the upper bound of $s_r(2)$ above,
so we can also apply the induction on $Y$).

Since $\Pic W$ is spanned, over $\BQQ$, by $H_W$ and the pullback of $\Pic Y$,
we have ${f_r^*}_{ | \Pic W} = q \, \id_{\Pic W}$ and hence ${f^*}_{ | \Pic X} = q \, \id_{| \Pic X}$
as reasoned in the case $\dim Y = 0$
(cf. the (**) above and the proof of Lemma \ref{surf}).
Since $s_r(2) = \dim W_0 + 1$, applying the proof of Proposition \ref{ThF} to the pair
$(W_0, f_{W_0})$ and noting that $f_{W_0}^* K_{W_0} \sim_{\BQQ} qK_{W_0}$
for $K_{W_0} = {K_W}_{ | W_0}$, we have
$$(***) \,\,\,\,\, K_{W_0} + \sum_{\ell=1}^{s_r} {V(r)_{\ell}}_{ | W_0} \equiv 0 .$$
Using this and restricting the display (*) above to $W_0$, we get
$\Delta(r)_{| W_0} = 0$ so that $\Delta(r) = \pi^* \Delta_Y$ for some effective divisor
$\Delta_Y \subset Y$.

We now show that $\Delta(r) = 0$ (i.e., $\Delta_Y = 0$).
Let
$$\sigma(s_r) : W(s_r) \to W$$
be the normalization of $V(r)_{s_r} \subset W$.
Pulling back the display (*) above, we get (with $k = s_r$):
$$(****) \,\,\,\,\, K_{W(k)} + D(k) = f(k)^*(K_{W(k)} + D(k)) + \Delta(r)_k$$
where $f(k)$ is the lifting of ${f_k}_{|V(r)_{k}}$
(polarized by the pullback of $H_W$), and $\Delta(r)_k$ is the
pullback of $\Delta(r)$.
As in the proof of Proposition \ref{ThF}, $(W(s_r), D(s_r))$ is log canonical with a reduced divisor
$$D(s_r) \ge \sum_{t = 1}^{s_r - 1} \sigma(s_r)^* V(r)_t$$
where each term in the summand with $t > s_r(1)$ is nonzero because
${V(r)_{t'}}_{| W_0}$ ($t'= t, r$) are ample
and $\sigma(s_r)$ is finite over $V(r)_{s_r} \cap W_0$.
Next, consider the normalization
$$\sigma(k) : W(k) \to W(k+1)$$
(with $k = s_r-1$)
of an irreducible component of $\sigma(k+1)^* V(r)_{k} \subset W_{k+1}$.
Thus for
$$k = s_r, \, s_r-1, \, \dots, \,
k_0 := s_r - \dim W_0 + 1 = s_r(1) + 2$$
we have normalizations $$\sigma(k) : W(k) \to W(k+1) ,$$ log canonical pairs $(W(k), D(k))$
with a reduced divisor
$$D(k) \, \ge \, \text{(the pullback of} \,\, \sum_{i=1}^{k-1} V(r)_i),$$
and the display (****) above
for all these $k$.
The natural composition $\tau_{k_0} : W(k_0) \to Y$
is generically finite and surjective.
Pushing forward the display (****) above with $(\pi_{k_0}/(\deg(\tau_{k_0}))$
and noting that $V(r)_i = \pi^*G_i$ ($1 \le i \le s_r(1)$) and $\Delta(r) = \pi^*\Delta_Y$,
we get
$$K_Y + \sum_{i=1}^{s_r(1)} G_i + C = f_Y^*(K_Y + \sum_{i=1}^{s_r(1)} G_i + C) + \Delta_Y$$
where $C$ is an effective divisor contributed from the branch locus of $\tau_{k_0}$
and others. This and the display (**) above imply
$0 + C \equiv f_Y^*C + \Delta_Y$
and
$0 \equiv (q-1)C + \Delta_Y$.
Thus $C = 0 = \Delta_Y$. Hence $\Delta(r) = 0$.

Now $K_W + \sum_{\ell=1}^{n+\rho(X)} V(r)_{\ell}$ restricts to zero on $W_0$ by the (***) above,
so it is $\BQQ$-linearly equivalent to some pullback $\pi^*L$. Thus the display (*) above becomes
$$\pi^*L \sim_{\BQQ} f_r^*\pi^*L + 0 = \pi^* f_Y^*L \sim_{\BQQ} q \pi^*L .$$
Hence $0 \sim_{\BQQ} \pi^* L \sim K_W + \sum_{\ell} V(r)_{\ell}$.
This will deduce Theorem \ref{ThG}(4) as
we did in the case $\dim Y = 0$.

For Theorem \ref{ThG}(2), by the above induction, either two general points of
$X$ are connected by a chain of rational curves and hence $X$ is rationally connected
by \cite[Corollary 1.5]{HM},
or there is a fibration $\eta: X \to E$ onto a smooth projective curve $E$
such that some power $f^k$ descends to some
$f_E : E \to E$ of degree $q > 1$ (and with genus $g(E) \ge 1$, so $g(E) = 1$)
and a general fibre $X_e$ is rationally
connected; here we used again \cite{GHS}.
In the latter case, let $\Sigma := \{e \in E \, | \, X_e$ is not rationally connected$\}$.
Then $f_E^{-1}(\Sigma) \subseteq \Sigma$. Applying $f_E^{-1}$ a few times
and comparing the cardinalities
of the sets involved, we see that $f_E^{-1}(\Sigma) = \Sigma$.
So $\Sigma = \emptyset$ (and hence every fibre $X_e$ is rationally connected) since $f_E$ is \'etale.
Similarly, every fibre $X_e$ is irreducible and normal
(cf. \cite[Lemma 4.7]{nz2}).
This proves Theorem \ref{ThG}.
\end{setup}


\end{document}